\documentclass[leqno,11pt]{amsart}
\usepackage{amssymb}

\newcommand{\ulim}{\mathrm u\mbox{-}\kern-2pt\varinjlim}
\newcommand{\uLlim}{\mathrm u^L\mbox{-}\kern-2pt\varinjlim}

\newcommand{\tlim}{\mathrm t\mbox{-}\kern-2pt\varinjlim}
\newcommand{\glim}{\mathrm g\mbox{-}\kern-2pt\varinjlim}
\newcommand{\w}{\omega}

\newcommand{\LP}{\operatornamewithlimits{\overrightarrow{\textstyle\prod}}}
\newcommand{\RP}{\operatornamewithlimits{\overleftarrow{\textstyle\prod}}}
\newcommand{\RLP}{\operatornamewithlimits{\overleftrightarrow{\textstyle\prod}}}
\newcommand{\LR[1]}{\overset{\rightleftharpoons}{#1}}
\newcommand{\LA[1]}{\overset{\rightarrow}{#1}}
\newcommand{\RA[1]}{\overset{\leftarrow}{#1}}
\newcommand{\RL[1]}{\overset{\leftrightarrow}{#1}}
\newcommand{\Ra}{\Rightarrow}
\newcommand{\LRs}{\mathrm{LR}}
\newcommand{\RLs}{\mathrm{RL}}
\newcommand{\Rs}{\mathrm{R}}
\newcommand{\Ls}{\mathrm{L}}
\newcommand{\tprod}{\operatornamewithlimits{\textstyle{\prod}}}

\newcommand{\U}{\mathcal U}
\newcommand{\IN}{\mathbb N}

\newcommand{\PTA}{\mathrm{PTA}}
\newcommand{\HH}{\mathcal H}
\newcommand{\IR}{\mathbb R}
\newcommand{\cl}{\mathrm{cl}}
\newcommand{\supp}{\mathrm{supp}}
\newcommand{\II}{\mathbb I}
\newcommand{\id}{\mathrm{id}}
\newcommand{\e}{\varepsilon}
\newcommand{\BB}{\mathcal B}
\newcommand{\PM}{\mathcal{P\!M}}
\newcommand{\A}{\mathcal A}
\newcommand{\dlim}{\varinjlim}
\newcommand{\mprod}{\operatornamewithlimits{\bigwedge}}
\newcommand{\rozd}{\, ;\;}

\theoremstyle{plain} %text of this environment is typesetted in italics
\newtheorem{theorem}{\indent\sc Theorem}[section]

\newtheorem{proposition}[theorem]{\indent\sc Proposition}
\theoremstyle{definition} %text of this environment is typesetted in roman letters
\newtheorem{definition}[theorem]{\indent\sc Definition}
\newtheorem{remark}[theorem]{\indent\sc Remark}
\newtheorem{example}[theorem]{\indent\sc Example}
\newtheorem{problem}[theorem]{\indent\sc Problem}

\begin{document}

\title[Direct limits of topological groups]{Direct limit topologies in the categories of topological groups and of uniform spaces}

\author[T.~Banakh]{Taras Banakh} %first author's name and the running head option

\author[D.~Repov\v s]{Du\v san Repov\v s} %second author's name and the running head option



\address{%Department of Mathematics, Ivan Franko National University of Lviv, and\newline
Instytut Matematyki, Uniwersytet Humanistyczno-Przyrodniczy Jana Kochanowskiego,  \'Swi\c etokrzyska 15, Kielce, Poland}

\email{tbanakh@yahoo.com}

\address{Faculty of Mathematics and Physics, 
% and Faculty of Education,
University of Ljubljana,
Jadranska 19,
Ljubljana, Slovenia 1000}
\email{dusan.repovs@guest.arnes.si}

\thanks{This research was supported by Slovenian Research Agency grant P1-0292-0101, J1-9643-0101 and 
J1-2057-0101.}

\maketitle

\begin{abstract} Given an increasing sequence  $(G_n)$ of topological groups, we study the topologies of the direct limits of the sequence $(G_n)$ in the categories of topological groups and of uniform spaces and find conditions under which these two direct limit topologies coincide.
\end{abstract}

\section{Introduction}

Given a tower
$$G_0\subset G_1\subset G_2\subset\cdots
$$of topological groups, we study
in this paper the topological structure of the direct limit $\glim G_n$ of the tower $(G_n)$ in the category of topological groups. By definition, $\glim G_n$ is the union $G=\bigcup_{n\in\w} G_n$ endowed with the strongest (not necessarily Hausdorff) topology that turns $G$ into a topological group and makes the identity inclusions $G_n\to G$, $n\in\w$, continuous. 

Besides the topology of $\glim G_n$, the union $G=\bigcup_{n\in\w}G_n$ carries the topology of the direct limit $\tlim G_n$ of the tower $(G_n)_{n\in\w}$ in the category of topological spaces. The topology of $\tlim G_n$ is the strongest topology on $G$ making the identity inclusions $G_n\to G$, $n\in\w$, continuous.

The definitions of the direct limits $\glim G_n$ and $\tlim G_n$ imply that the identity map
$$\tlim G_n\to\glim G_n$$is continuous. This map is a homeomorphism if and only if $\tlim G_n$ is a topological group. It was observed in \cite{Ba98} and \cite{TSH}
that the group operation on $G=\tlim G_n$ is not necessarily continuous with respect to the topology $\tlim G_n$. Moreover, if each group $G_n$, $n\in\w$, is metrizable and closed in $G_{n+1}$, then the topological direct limit $\tlim G_n$ is a topological group if and only if either all groups $G_n$ are locally compact or 
some group $G_n$ is open in all groups $G_m$, $m\ge n$ (see \cite{BaZ} or \cite{Yama}).

Thus in many interesting cases (in particular, those considered in
\cite{Dierolf}, \cite{Floret}, \cite{Glo03}, \cite{Glo06}, \cite{Glo08}),
the topology of $\glim G_n$ differs from the topology of the topological direct limit $\tlim G_n$. However, in contrast with the topology of $\tlim G_n$ which has an explicit description (as the family of all subsets $U\subset \bigcup_{n\in\w}G_n$ that have open traces $U\cap G_n$ on all spaces $G_n$) the topological structure of the direct limit $\glim G_n$ is not so clear. The problem of explicit description of the topological structure of the direct limit $\glim G_n$ was discussed in 
\cite{Eda}, 
\cite{HSTH},
\cite{Glo03}, 
\cite{Glo06}, 
\cite{Glo08},
\cite{TSH}.

In this paper we shall show that under certain conditions on a tower of topological groups $(G_n)_{n\in\w}$ the topology  of the direct limit $\glim G_n$ coincides with one (or all) of four simply described topologies $\LA[\tau]$, $\RA[\tau]$, $\LR[\tau]$ or $\RL[\tau]$ on the group $G=\bigcup_{n\in\w}G_n$. These topologies are considered in 
Section~\ref{s:top}. In Sections~\ref{s:PTA}
and
\ref{s1:SIN} we study two properties ($\PTA$ and the balanced property) of a tower of topological groups $(G_n)_{n\in\w}$ implying that the topology of $\glim G_n$ coincides with the topology $\LR[\tau]$, which is the strongest among the four topologies on $G$. In Section~\ref{s2:SIN} we define another (bi-balanced) property of the tower $(G_n)$ guaranteeing that the topology of $\glim G_n$ coincides with the topology $\RL[\tau]$, which is the weakest among the four topologies on $G$. In Section~\ref{s:ulim-top} we reveal the uniform nature of the topologies $\LA[\tau]$ and $\RA[\tau]$ and show that they coincide with the topologies of the uniform direct limits $\ulim G_n^{\Ls}$ and $\ulim G_n^R$ of the groups $G_n$ endowed with the left and right uniformities. In Section~\ref{s:OP} we sum up the results obtained in this paper and pose some open problems.

\section{The semitopological groups $\LA[G]$, $\RA[G]$, $\RL[G]$ and $\LR[G]$}\label{s:top} 

In this section, given a tower of topological groups
$$G_0\subset G_1\subset G_2\subset\cdots
$$we define four  topologies $\LA[\tau]$, $\RA[\tau]$, $\RL[\tau]$ or $\LR[\tau]$ on the group $G=\bigcup_{n\in\w}G_n$. 

Given a sequence of subsets $(U_n)_{n\in\w}$ of the group $G$, consider their directed products in $G$:
$$\begin{aligned}
&\LP_{n\in\w}U_n=\bigcup_{m\in\w}\LP_{0\le n\le m}U_n\mbox{ \ where \ }\LP_{k\le n\le m}U_n=U_k U_{k+1}\cdots U_m,\\
&\RP_{n\in\w}U_n=\bigcup_{m\in\w}\RP_{0\le n\le m}U_n\mbox{ \ where \ }\RP_{k\le n\le m}U_n=U_m\cdots U_{k+1}U_k,\\
&\RLP_{n\in\w}U_n=\bigcup_{m\in\w}\RLP_{0\le n\le m}U_n\mbox{ where }
\RLP_{k\le n\le m}U_n=U_m\cdots U_kU_k\cdots U_m.
\end{aligned}
$$
Observe that $$\big(\LP_{n\in\w}U_n\big)^{-1}=\RP_{n\in\w}U_n^{-1}\mbox{ and }\RLP_{n\in\w}U_n=
\big(\RP_{n\in\w}U_n\big)\cdot\big(\LP_{n\in\w}U_n\big).$$

In each topological group $G_n$ fix a base $\BB_n$ of open symmetric neighborhoods $U=U^{-1}\subset G_n$ of the neutral element $e$. 

The topologies $\LA[\tau]$, $\RA[\tau]$, $\RL[\tau]$ and $\LR[\tau]$ on the group $G=\bigcup_{n\in\w}G_n$ are generated by the bases:
$$
\begin{aligned}
&\LA[\BB]=\{(\LP_{n\in\w}U_n\big)\cdot x \rozd     x\in G,\;(U_n)_{n\in\w}\in\tprod\limits_{n\in\w}\BB_n\},\\
&\RA[\BB]=\{x\cdot (\RP_{n\in\w}U_n\big) \rozd   x\in G,\;(U_n)_{n\in\w}\in\tprod_{n\in\w}\BB_n\},\\
&\RL[\BB]=\{x\cdot(\RLP_{n\in\w}U_n\big)\cdot y\rozd    x,y\in G,\;(U_n)_{n\in\w}\in\tprod_{n\in\w}\BB_n\},\\
&\LR[\BB]=\{x\big(\RP_{n\in\w}xU_n\big)\cap \big(\LP_{n\in\w}U_n\big)y\rozd    x,y\in G,\;(U_n)_{n\in\w}\in\tprod_{n\in\w}\BB_n\}.
\end{aligned}
$$
By $\LA[G]$, $\RA[G]$, $\RL[G]$, $\LR[G]$ we denote the groups $G$ endowed with the topologies 
$\LA[\tau]$, $\RA[\tau]$, $\RL[\tau]$, $\LR[\tau]$, respectively. It is easy to check that 
$\LA[G]$, $\RA[G]$, $\RL[G]$, $\LR[G]$ are semitopological groups having the families
$$
\begin{aligned}
&\LA[\BB]_e=\{\LP_{n\in\w}U_n\rozd (U_n)_{n\in\w}\in\prod_{n\in\w}\BB_n\},\\
&\RA[\BB]_e=\{\RP_{n\in\w}U_n\rozd (U_n)_{n\in\w}\in\prod_{n\in\w}\BB_n\},\\
&\RL[\BB]_e=\{\RLP_{n\in\w}U_n\rozd (U_n)_{n\in\w}\in\prod_{n\in\w}\BB_n\}=\{U^{-1} U\rozd U\in\LA[\BB]_e\},\\
&\LR[\BB]_e=\{\big(\LP_{n\in\w}U_n\big)\cap \big(\RP_{n\in\w}U_n\big)\rozd (U_n)_{n\in\w}\in\prod_{n\in\w}\BB_n\}=\{U\cap  U^{-1}\rozd U\in\LA[\BB]_e\}
\end{aligned}
$$as neighborhood bases at the identity $e$. Since the inversion $(\cdot)^{-1}:G\to G$ is continuous with respect to the topologies $\RL[\tau]$ or $\LR[\tau]$, the semitopological groups $\RL[G]$ and $\LR[G]$ are quasitopological groups.

We recall that a group $H$ endowed with a topology is 
\begin{itemize}
\item a {\em semitopological group} if the binary operation $H\times H\to H$, $(x,y)\mapsto xy,$ is separately continuous;
\item a {\em quasitopological group} if $H$ is a semitopological group with continuous inversion $(\cdot)^{-1}:H\to H$, $(\cdot)^{-1}:x\mapsto x^{-1}$.
\end{itemize}

Now we see that for any tower $(G_n)_{n\in\w}$ of topological groups we get the following five semitopological groups linked by continuous identity homomorphisms:

\begin{picture}(200,90)(-100,-10)
\put(19,29){$\LR[G]$}
\put(33,28){\vector(3,-2){30}}
\put(33,38){\vector(3,2){30}}
\put(69,2){$\LA[G]$}
\put(85,8){\vector(3,2){30}}
\put(69,58){$\RA[G]$}
\put(85,56){\vector(3,-2){30}}
\put(120,29){$\RL[G]$}
\put(135,32){\vector(1,0){27}}
\put(170,30){$\glim G_n$}
\end{picture}

The continuity of the final map in the diagram is not trivial: 

\begin{proposition}\label{p2.1} The identity map $\RL[G]\to\glim G_n$ is continuous.
\end{proposition}

\begin{proof} Since $\RL[G]$ and $\glim G_n$ are semitopological groups, it suffices to prove the continuity of the identity map $\RL[G]\to\glim G_n$ at the neutral element $e$.

Given a neighborhood $U\subset\glim G_n$ of $e$, find an open neighborhood $V\subset\glim G_n$ of $e$ such that $V^{-1}V\subset U$. Such a neighborhood exists because $\glim G_n$ is a topological group. By induction, construct a sequence of open symmetric neighborhoods $V_n\subset\glim G_n$ of $e$ such that $V_0^2\subset V$ and $V_{n+1}^2\subset V_n$ for all $n\in\w$. By induction on $m\in\w$ we shall prove the inclusion 
\begin{equation}\label{ind1}
\big(\LP_{0\le n<m}V_n\big)\cdot V_m^2\subset V.
\end{equation}
For $m=0$ this inclusion holds according to the choice of $V_0$. Assuming that for some $m$ the inclusion is true observe that
$$\big(\LP_{0\le n\le m}V_n\big)\cdot V_{m+1}^2=\big(\LP_{0\le n<m}V_n\big)\cdot V_mV_{m+1}^2\subset 
\big(\LP_{0\le n<m}V_n\big)\cdot V_mV_m\subset V$$by the inductive hypothesis.
Then $$\LP_{n\in\w}V_n=\bigcup_{m\in\w}\LP_{0\le n\le m}V_n\subset V.$$

For every $n\in\w$ find a basic neighborhood $W_n\in\BB_n$ in the group $G_n$ such that $W_n\subset V_n$ and observe that
$\LP_{n\in\w}W_n\subset\LP_{n\in\w}V_n\subset V$ and hence
$$\RL[\BB]_e\ni\RLP_{n\in w}W_n=\big(\LP_{n\in\w}W_n\big)^{-1}\cdot \LP_{n\in\w}W_n\subset V^{-1}V\subset U$$witnessing the continuity of the identity map $\RL[G]\to\glim G_n$ at $e$. 
\end{proof}

One may ask about conditions guaranteeing that the semitopological groups $\LA[G]$, $\RA[G]$, $\LR[G]$ or $\RL[G]$ are topological groups.

\begin{theorem}\label{t2.2} The following conditions (1) through (5) are equivalent\textup{:}
\begin{enumerate}
\item $\LA[G]$ is a topological group\textup{;}
\item $\RA[G]$ is a topological group\textup{;}
\item $\LR[G]$ is a topological group\textup{;}
\item the identity map $\RL[G]\to\LR[G]$ is continuous\textup{;}
\item the identity map $\LR[G]\to \glim G_n$ is a homeomorphism.
\end{enumerate}
The equivalent conditions (1) through (5) imply the following two equivalent conditions\textup{:}
\begin{enumerate}
\item[(6)] $\RL[G]$ is a topological group\textup{;}
\item[(7)] the identity map $\RL[G]\to \glim G_n$ is a homeomorphism.
\end{enumerate}
\end{theorem}

\begin{proof} $(1)\Ra(2)$ Assume that $\LA[G]$ is a topological group. Then the identity map $\LA[G]\to\RA[G]$ is continuous because each basic neighborhood $\RP_{n\in\w}U_n\in\RA[\BB]_e$ of $e$ in $\RA[G]$ is open in $\LA[G]$, being the inversion $(\LP_{n\in\w}U_n\big)^{-1}$ of the basic neighborhood $\LP_{n\in\w}U_n\in\LA[\BB]_e$ of $e$ in the topological group $\LA[G]$. By the same reason, the identity map $\LA[G]\to\RA[G]$ is open. Consequently, the topologies $\LA[\tau]$ and $\RA[\tau]$ on $G$ coincide, and hence $\LR[G]$ is a topological group.
\smallskip

The implication $(2)\Ra(1)$ can be proved by analogy.
\smallskip

$(1)\Ra(3)$ If $\LA[G]$ is a topological group, then $\LA[\tau]=\RA[\tau]$ and then $\LR[\tau]=\LA[\tau]=\RA[\tau]$, by the definition of the topology $\LR[\tau]$. Consequently, $\LR[G]=\LA[G]$ is a topological group.
\smallskip

$(3)\Ra(5)$ If $\LR[G]$ is a topological group, then the identity map $\glim G_n\to\LR[G]$ is continuous by the definition of $\glim G_n$ because all the identity homomorphisms $G_n\to \LR[G]$, $n\in\w$, are continuous. 
The inverse (identity) map $\LR[G]\to\glim G_n$ is always continuous by Proposition~\ref{p2.1}. So, it is a homeomorphism.

$(5)\Ra(4)$ If the identity map $\LR[G]\to \glim G_n$ is a homeomorphism, then the identity map $\RL[G]\to\LR[G]$ is continuous being the composition of two continuous maps $\RL[G]\to\glim G_n\to\LR[G]$.

$(4)\Ra(1)$ Assume that the identity map $\RL[G]\to\LR[G]$ is a homeomorphism. Then the identity maps 
between the semitopological groups $\LR[G],\LA[G],\RA[G],\RL[G]$ are homeomorphisms. Consequently, $\LA[G]$ is a quasitopological group because so is $\LR[G]$ or $\RL[G]$. To see that $\LA[G]$ is a topological group, observe that the multiplication map $\LA[G]\times\LA[G]\to\LA[G]$, $(x,y)\mapsto xy$, is continuous, being continuous as a map 
 $\RA[G]\times\LA[G]\to\RL[G]$.
\smallskip

$(5)\Ra(7)$ If the identity map $\LR[G]\to\glim G_n$ is a homeomorphism, then the identity map $\glim G_n\to\RL[G]$ is continuous being the composition of two continuous maps $\glim G_n\to\LR[G]\to\RL[G]$. The continuity of the inverse (identity) map $\RL[G]\to\glim G_n$ was proved in Proposition~\ref{p2.1}.

$(7)\Ra(6)$ If the identity map $\RL[G]\to\glim G_n$ is a homeomorphism, then $\RL[G]$ is a topological group because so is $\glim G_n$.

The final implication $(6)\Ra(7)$ can be proved by analogy with $(3)\Ra(5)$.
\end{proof}

\begin{remark} The topology $\RL[\tau]$ on the union $G=\bigcup_{n\in\w}G_n$ of a tower of topological groups $(G_n)$ was introduced in \cite{TSH} and called the {\em bamboo-shoot\/} topology. This topology was later discussed in \cite{Eda}, 
\cite{Glo03}, \cite{Glo06}, \cite{Glo08}, \cite{HSTH}.
\end{remark}

\section{The Passing Through Assumption}\label{s:PTA}
 
In this section we shall discuss implications of $\PTA$, the Passing Through Assumption, introduced by Tatsuuma, Shimomura, and Hirai in \cite{TSH}.

\begin{definition} A tower of topological groups $(G_n)_{n\in\w}$ is defined to satisfy $\PTA$ if each group $G_n$ has a neighborhood base $\mathcal B_n$ at the identity $e$, consisting of open symmetric neighborhoods $U\subset G_n$ such that for every  $m\ge n$ and every neighborhood $V\subset G_m$ of $e$ there is a neighborhood $W\subset G_m$ of $e$ such that $WU\subset UV$.
\end{definition}

It was proved in 
\cite{HSTH} and 
\cite{TSH}
that for a tower of topological groups $(G_n)_{n\in\w}$ satisfying $\PTA$, the semitopological group $\RL[G]$ is a topological group, which can be identified with the direct limit $\glim G_n$.

The following theorem says a bit more:

\begin{theorem}\label{t3.2} If a tower of topological groups $(G_n)_{n\in\w}$ satisfies $\PTA$, then the semitopological group $\LA[G]$ is a topological group and hence the conditions (1) through (7) of Theorem~\textup{\ref{t2.2}} hold. In particular the topology of $\glim G_n$ coincides with any of the topologies: $\LA[\tau]$, $\RA[\tau]$, $\LR[\tau]$, $\RL[\tau]$.
\end{theorem}

\begin{proof} Since the tower $(G_n)_{n\in\w}$ satisfies $\PTA$, each topological group $G_n$ admits a neighborhood base $\BB_n$ at $e$ that consists of open sets $U=U^{-1}$ such that for every $m\ge n$ and a neighborhood $V\subset G_m$ of $e$ there is a neighborhood $W\subset G_m$ of $e$ such that $WU\subset UV$.

In order to show that the semitopological group $\LA[G]$ is a topological group, it suffices to check the continuity of the multiplication and of the inversion at the neutral element $e$.
\smallskip

The continuity of the multiplication at $e$ will follow as soon as for every neighborhood $\LP_{n\in\w}W_n\in\LA[\BB]_e$ we find a neighborhood $\LP_{n\in\w}V_n\in\LA[\BB]_e$ such that $\big(\LP_{n\in\w}V_n\big)^2\subset \big(\LP_{n\in\w}W_n\big)$.

For every $n\in\w$ find a neighborhood $U_n\in\BB_n$ with $U_nU_n\subset W_n$. Put 
$V_n^{(0)}=U_n$ and using PTA, for every $0<i\le n$ find a neighborhood $V_n^{(i)}\in\BB_n$ such that
\begin{itemize}
\item $V_n^{(i)}\subset U_n$;
\item $V_n^{(i)}U_{n-i}\subset U_{n-i} V_n^{(i-1)}$.
\end{itemize} Observe that for $i=n-k$ the latter inclusion yields
\begin{equation}\label{PTA2a}
V_n^{(n-k)}U_k\subset U_kV_n^{(n-k-1)}.
\end{equation}

We claim that $\big(\LP_{n\in\w}V_n^{(n)}\big)^2\subset \big(\LP_{n\in\w}W_n\big)$.
Since $V_n^{(n)}\subset U_n$, this inclusion will follow as soon as we check that
\begin{equation}\label{PTA2c}
\LP_{n\le m}V_n^{(n)}\cdot \LP_{n\le m}U_n\subset\LP_{n\le m}W_n
\end{equation} 
for every $m>0$.

For every non-negative integer $k\le m+1$ consider the subset
$$\Pi_k=\LP_{0\le n<k}W_n\cdot\LP_{k\le n\le m}V_n^{(n-k)}\cdot\LP_{k\le n\le m}U_n$$of the group $G_m$. Observe that (\ref{PTA2c}) is equivalent to  the inclusion $\Pi_0\subset\Pi_{m+1}$. The last inclusion will follow as soon as we check that $\Pi_{k}\subset \Pi_{k+1}$ for every $k\le  m$.

By induction on $k$ we can deduce from (\ref{PTA2a}) the following inclusion:
\begin{equation}
\label{PTA2d}
(\LP_{k<n\le m}V_n^{(n-k)})\cdot U_k\subset U_k\cdot\LP_{k<n\le m}V_n^{(n-k-1)}.
\end{equation}

This inclusion combined with $V_k^{(0)}U_k=U_kU_k\subset W_k$ yields the desired inclusion:
$$
\begin{aligned}
\Pi_{k}&=\LP_{0\le n<k} W_n\cdot \LP_{k\le n\le m}V_n^{(n-k)}\cdot \LP_{k\le n\le m}U_n=\\
&=\Big(\LP_{0\le n<k} W_n\Big)\cdot V_k^{(0)}\cdot \Big(\LP_{k<n\le m}V_n^{(n-k)}\Big)\cdot U_k\cdot 
\LP_{k<n\le m}U_n\\
&\subset\Big(\LP_{0\le n<k}W_n\Big)\cdot V_k^{(0)}\cdot \Big(U_k\cdot\LP_{k<n\le m}V_n^{(n-k-1)}\Big)\cdot\LP_{k<n\le m}U_n\\
&\subset\Big(\LP_{0\le n<k}W_n\Big)\cdot W_k\cdot \LP_{k<n\le m}V_n^{(n-k-1)}\cdot\LP_{k<n\le m}U_n=\Pi_{k+1}.
\end{aligned}
$$
\smallskip

Next, we verify the continuity of the inversion at $e$.
Given a set $\LP_{n\in\w}U_n\in\LA[\BB]_e$, we need to find a set 
$\LP_{n\in\w}V_n\in\LA[\BB]_e$ such that $\Big(\LP_{n\in\w}V_n\Big)^{-1}\subset\LP_{n\in\w}U_n$. 

For every $n\in\w$ put $V_n^{(0)}=U_n$ and using $\PTA$, for every $0<i\le n$ choose a neighborhood $V_n^{(i)}\in\BB_n$ such that $V_n^{(i)}U_{n-i}\subset U_{n-i}V_n^{(i-1)}$. The so-defined sets satisfy the inclusions
\begin{equation}\label{PTA3a}
V_n^{(n-k)}U_k\subset U_kV_n^{(n-k-1)},\quad 0\le k<n.
\end{equation}

We claim that $$\big(\LP_{n\in\w}V_n^{(n)}\Big)^{-1}\subset \LP_{n\in\w}U_n.$$This inclusion will follow as soon as we check that 
\begin{equation}\label{PTA3b}
 \big(\LP_{n\le m}V_n^{(n)}\Big)^{-1}=\RP_{n\le m}V_n^{(n)}\subset \LP_{n\le m}U_n
\end{equation}
for all $m\in\w$. The left-hand equality follows from the symmetry of the neighborhoods $V_n^{(n)}\in\BB_n$.

For the proof of the right-hand inclusion, for every $k\le m+1$ consider the subset
$$\Pi_k=\LP_{0\le n<k}U_n\cdot\RP_{k\le n\le m}V_n^{(n-k)}$$ of the group $G_m$, and observe that 
(\ref{PTA3b}) is equivalent to the inclusion $\Pi_0\subset\Pi_{m+1}$. So it suffices to check that $\Pi_k\subset\Pi_{k+1}$ for every $k\le m+1$. 

By induction on $k\le m+1$ we can show that (\ref{PTA3a}) implies
$$\Big(\RP_{k<n\le m}V_n^{(n-k)}\Big)\cdot U_k\subset U_k\cdot \RP_{k<n\le m}V_n^{(n-k-1)}.$$

Now the inclusion $\Pi_k\subset\Pi_{k+1}$ can be seen as follows:
$$
\begin{aligned}
\Pi_k&=\LP_{0\le n<k}U_n\cdot\RP_{k\le n\le m}V_n^{(n-k)}=
\LP_{0\le n<k}U_n\cdot \Big(\RP_{k<n\le m}V_n^{(n-k)}\Big)\cdot V_k^{(0)}=\\
&=\LP_{0\le n<k}U_n\cdot \Big(\RP_{k<n\le m}V_n^{(n-k)}\Big)\cdot U_k\subset\Big(\LP_{0\le n<k}U_n\Big)\cdot U_k\cdot \RP_{k<n\le m}V_n^{(n-k-1)}=\Pi_{k+1}.
\end{aligned}
$$
\end{proof}

\section{Balances triples and towers of groups}\label{s1:SIN}

In this section we introduce another condition guaranteeing that the topology of the direct limit $\glim G_n$ of a tower  of topological  groups $(G_n)_{n\in\w}$ coincides with the topologies $\LA[\tau]$, $\RA[\tau]$, $\LR[\tau]$ and $\RL[\tau]$.
\smallskip

Let us observe that a tower of topological groups $(G_n)_{n\in\w}$ satisfies $\PTA$ if each group $G_n$, $n\in\w$, is {\em balanced}. The latter means that $G_n$ has a neighborhood base at $e$ consisting of $G$-invariant neighborhoods, see \cite[p.69]{AT}. 

We define a subset $U\subset G$ of a group $G$ to be {\em $H$-invariant} for a subgroup $H\subset G$ if $xUx^{-1}=U$ for all $x\in H$. Observe that for any subset $U\subset G$ the set
$$\sqrt[H]{U}=\{x\in G\rozd x^H\subset U\}$$ is the largest $H$-invariant subset of $U$. Here $x^H=\{hxh^{-1}\rozd h\in H\}$ stands for the conjugacy class of a point $x\in G$. 

Observe that a topological group $G$ is balanced if and only if for every neighborhood $U\subset G$ of $e$ the set $\sqrt[G]{U}$ is a neighborhood of $e$.

\begin{definition} A triple $(G,\Gamma,H)$ of topological groups $H\subset \Gamma\subset G$ is called {\em balanced} if for any neighborhoods $V\subset \Gamma$ and $U\subset G$ of the neutral element $e$ of $G$ the product $V\cdot\sqrt[H]{U}$ is a neighborhood of $e$ in $G$. 

A tower of topological groups $(G_n)_{n\in\w}$ is called {\em balanced\/} if each triple  $(G_{n+2},G_{n+1},G_n)$, $n\in\w$, is balanced.
\end{definition}

\begin{theorem}\label{t4.2} If a tower of topological groups $(G_n)_{n\in\w}$ is balanced, then the semitopological group $\LA[G]$ is a topological group and hence all the conditions (1) through (7) of Theorem~\textup{\ref{t2.2}} hold. In particular the topology of $\glim G_n$ coincides with any of the topologies: $\LA[\tau]$, $\RA[\tau]$, $\LR[\tau]$, $\RL[\tau]$.
\end{theorem}

 \begin{proof} In order to show that the semitopological group $\LA[G]$ is a topological group, it suffices to check the continuity of the multiplication and of the inversion at the neutral element $e$.
\smallskip

In order to check the continuity of multiplication at $e$, fix a neighborhood $\LP_{n\in\w}U_n\in\LA[\BB]_e$.
For every $n\in\w$, find a symmetric neighborhood $W_n$ of $e$ in the group $G_n$ such that $W_n\cdot W_n\subset U_n$ and let $$Z_n=\sqrt[G_{n-2}]{W_n}=\{x\in G_n\rozd x^{G_{n-2}}\subset W_n\}$$ be the largest $G_{n-2}$-invariant subset of $W_n$ (here we assume that $G_{k}=\{e\}$ for $k<0$). 

Let $V_0=U_0\cap W_1$ and $V_1\subset G_1$ be a symmetric neighborhood of $e$ such that $V_1^2\subset W_1$. Next, for each $n\ge 2$ by induction choose a neighborhood $V_n\subset G_n$ so that
\begin{itemize}
\item[(a)] $V_{n}^2\subset V_{n-1}\cdot Z_{n}$, and
\item[(b)] $V_{n}\subset W_{n+1}$,
\end{itemize}
The condition (a) can be satisfied because the triple $(G_{n},G_{n-1},G_{n-2})$ is  balanced according to our hypothesis.

We claim $\Big(\LP_{n\in\w} V_n\Big)^2\subset\LP_{n\in\w}U_n$. This inclusion will follow as soon as we check that 
$$(\LP_{n\le m}V_n)\cdot (\LP_{n\le m}V_n)\subset\LP_{n\le m+1}U_n$$
for every $m\ge 2$.

For every $1\le k\le m$ consider the subset
$$\Pi_k=(\LP_{n\le m-k}V_n)\cdot V_{m-k+1}^2\cdot(\LP_{n\le m-k}V_n)\cdot (\LP_{m-k< n<m} Z_{n+1}V_n)\cdot V_m.$$

We claim that $\Pi_k\subset\Pi_{k+1}$. Indeed,
$$
\begin{aligned}
\Pi_k&=\Big(\LP_{n\le m-k}V_n\Big)\cdot V_{m-k+1}^2\cdot\Big(\LP_{n\le m-k}V_n\Big)\cdot \Big(\LP_{m-k< n<m} Z_{n+1}V_n\Big)\cdot V_m   \subset\\
&\Big(\LP_{n\le m-k}V_n\Big)\cdot V_{m-k}\cdot Z_{m-k+1}\cdot\Big(\LP_{n<m-k}V_n\Big)\cdot V_{m-k}\cdot \Big(\LP_{m-k< n<m} Z_{n+1}V_n\Big)\cdot V_m=\\
&\Big(\LP_{n<m-k}V_n\Big)\cdot V^2_{m-k}\cdot\Big(\LP_{n<m-k}V_n\Big)\cdot Z_{m-k+1}\cdot V_{m-k}\cdot\Big(\LP_{m-k<n<m}Z_{n+1}V_n\Big)\cdot V_m=\\
&\Big(\LP_{n<m-k}V_n\Big)\cdot V^2_{m-k}\cdot\Big(\LP_{n<m-k}V_n\Big)\cdot \Big(\LP_{m-k\le n<m}Z_{n+1}V_n\Big)\cdot V_m=
\Pi_{k+1}.
\end{aligned}
$$
Now we see that
$$\begin{aligned}
\big(\LP_{i\le m}V_i\big)^2&\subset(\LP_{i\le m-1}V_i\big)\cdot V^2_m\cdot  \big(\LP_{i\le m}V_i\big)=\Pi_1\subset\Pi_m=V_0V_1^2V_0\cdot\big(\LP_{0<n<m}Z_{n+1}V_n\big)V_m\subset\\
&U_0W_1W_1\cdot\big(\LP_{0<n<m}W_{n+1}W_{n+1}\big)W_{m+1}\subset U_0U_1\big(\LP_{0<n<m}U_{n+1}\big)\cdot U_{m+1}=\LP_{n\le m+1}U_n.
\end{aligned}
$$

Now we check that the inversion is continuous at $e$ with respect to the topology $\LA[\tau]$. Given any basic set $\LP_{n\in\w}W_n\in\LA[\BB]_e$, we need to find a basic set $\LP_{n\in\w}U_n\in\LA[\BB]_e$ such that $\big(\LP_{n\in\w}U_n\big)^{-1}\subset\LP_{n\in\w}W_n$. 

For every $n\in\w$ let $Z_{n+2}=\sqrt[G_{n}]{W_{n+2}}$ be the largest $G_n$-invariant subset of $W_{n+2}$. For each non-negative number $n<2$ pick a symmetric neighborhood $V_n\subset G_n$ such that $V_n^2\subset W_n$. For $n\ge 2$ by induction choose  a symmetric neighborhood $V_n\subset G_n$ of $e$ such that $V_n^2\subset W_n\cap(V_{n-1}\cdot Z_n)$. Such a neighborhood $V_n$ exists by the balanced property of the triple $(G_n,G_{n-1},G_{n-2})$. Finally, for every $n\in\w$ put  $U_n=V_n\cap V_{n+1}$.

We claim that $\big(\LP_{n\in\w}U_n\big)^{-1}\subset \LP_{n\in\w}W_n$. This inclusion will follow as soon as we check that $\RP_{n<m}U_n\subset \LP_{n\le m}W_n$ for every $m\in\w$. By induction we shall prove a bit more: 
\begin{equation}\label{eq:inv}
V_m\cdot\RP_{n<m}U_n\subset \LP_{n\le m}W_n
\end{equation} 
for every $m\in\IN$.  

For $m=1$ the inclusion (\ref{eq:inv}) is true: $V_1U_0\subset V_1^2\subset W_1\subset W_0W_1$. Assume that the inclusion (\ref{eq:inv}) has been proved for some $m=k\ge 1$.
Then $$
\begin{aligned}
V_{m+1}\cdot\RP_{n\le m}U_n&\subset V_{m+1}\cdot U_m\cdot\RP_{n<m}U_n\subset V_{m+1}^2\cdot\RP_{n<m}U_n\subset V_mZ_{m+1}\RP_{n<m}U_n=\\&=V_m\cdot \big(\RP_{n<m}U_n\big)\cdot Z_{m+1}\subset \big(\LP_{n\le m}W_n\big)\cdot Z_{m+1}\subset \LP_{n\le m+1}W_n,
\end{aligned}
$$which means that the inclusion (\ref{eq:inv}) holds for $m=k+1$.
\end{proof}

\section{Bi-balanced triples and towers of groups}\label{s2:SIN}

In this section we introduce the bi-balanced property of a tower $(G_n)$, which is weaker than the balanced property and implies that the semitopological group $\RL[G]$ is a topological group.

\begin{definition} A triple $(G,\Gamma,H)$ of topological groups $H\subset \Gamma\subset G$ is called {\em bi-balanced} if for any neighborhoods $V\subset \Gamma$ and $U\subset G$ of the neutral element $e$ of $G$ the product $\sqrt[H]{U}\cdot V\cdot\sqrt[H]{U}$ is a neighborhood of $e$ in $G$.

A tower of topological groups $(G_n)_{n\in\w}$ is called {\em bi-balanced\/} if each triple $(G_{n+2},G_{n+1},G_n)$, $n\in\w$, is bi-balanced.
\end{definition}

\begin{theorem}\label{t5.2} If a tower of topological groups $(G_n)_{n\in\w}$ is bi-balanced, then the identity map $\RL[G]\to\glim G_n$ is a homeomorphism and hence the topology of $\glim G_n$ coincides with the topology $\RL[\tau]$.
\end{theorem}

\begin{proof} By Theorem~\ref{t2.2}, it suffices to show that $\RL[G]$ is a topological group. Since $\RL[G]$ is a quasitopological group, it suffices to check the continuity of multiplication at the neutral element. Given a basic neighborhood $\RLP_{n\in\w}W_n\in\RL[\BB]_e$, we should find a neighborhood $\RLP_{n\in\w}V_n\in\RL[\BB]_e$ such that $\big(\RLP_{n\in\w}V_n\big)^2\subset \RLP_{n\in\w}W_n$. 

For every $n\in\w$, find a symmetric neighborhood $U_n\subset G_n$ of $e$ such that $U_n^2\subset W_n$ and let $Z_{n}=\sqrt[G_{n-2}]{U_{n}}$ be the maximal $G_{n-2}$-invariant subset of $U_{n}$ (here we assume that $G_k=\{e\}$ for $k<0$).
Let $\widetilde U_0=W_0$ and by induction for every $n\in\IN$ choose a  symmetric neighborhood $\widetilde U_n\subset U_n$ of $e$ such that $\widetilde U_n^3\subset Z_n\widetilde U_{n-1}Z_n$.
The choice of the neighborhood $\widetilde U_n$ is possible because the set $Z_n\widetilde U_{n-1}Z_n$ is a neighborhood of $e$ in the group $G_n$ by the bi-balanced property of the triple $(G_n,G_{n-1},G_{n-2})$. Finally, for every $n\in\w$ let $V_n=G_n\cap \widetilde U_{n+1}$.

We claim that $\RLP_{n\in\w}V_n$ is the required neighborhood with 
$\big(\RLP_{n\in\w}V_n\big)^2\subset \RLP_{n\in\w}W_n$. This inclusion will follow as soon as we check that 
\begin{equation}\label{SIN2b}
\big(\RLP_{n<m}V_n\big)^2\subset \RLP_{n\le m}W_n
\end{equation}
for all $m\in\w$.

For $m=0$ this inclusion is trivial. Assume that the inclusion (\ref{SIN2b}) has been proved for some $m=p\in\w$. We shall prove it for $m=p+1$. For every non-negative $k<m$ consider the subset
$$\Pi_{k}=\big(\RP_{k\le n<m}V_nZ_{n+1}\big)\cdot\big(\RLP_{n<k}V_n\big)\cdot \widetilde U_{k}\cdot \big(\RLP_{n< k}V_n\big)\cdot \big(\LP_{k\le n<m}Z_{n+1}V_n\big)$$
of the group $G_m$. The following chain of inclusions guarantees that $\Pi_{k+1}\subset\Pi_{k}$:
$$
\begin{aligned}
\Pi_{k+1}&=\big(\RP_{k<n<m}V_nZ_{n+1}\big)\cdot\big(\RLP_{n\le k}V_n\big)\cdot \widetilde U_{k+1}\cdot \big(\RLP_{n\le k}V_n\big)\cdot \big(\LP_{k<n<m}Z_{n+1}V_n\big)\subset\\
&\big(\RP_{k<n<m}V_nZ_{n+1}\big)\cdot V_k\cdot \big(\RLP_{n<k}V_n\big)\cdot V_k\cdot \widetilde U_{k+1}\cdot V_k\cdot \big(\RLP_{n<k}V_n\big)\cdot V_k\cdot  \big(\LP_{k<n<m}Z_{n+1}V_n\big)\subset\\
&\big(\RP_{k<n<m}V_nZ_{n+1}\big)\cdot V_k\cdot \big(\RLP_{n<k}V_n\big)\cdot \widetilde U_{k+1}^3\cdot \big(\RLP_{n<k}V_n\big)\cdot V_k\cdot  \big(\LP_{k<n<m}Z_{n+1}V_n\big)\subset\\
&\big(\RP_{k<n<m}V_nZ_{n+1}\big)\cdot V_k\cdot \big(\RLP_{n<k}V_n\big)\cdot Z_{k+1}\cdot\widetilde U_{k}\cdot Z_{k+1}\cdot \big(\RLP_{n<k}V_n\big)\cdot V_k\cdot  \big(\LP_{k<n<m}Z_{n+1}V_n\big)=\\
&\big(\RP_{k<n<m}V_nZ_{n+1}\big)\cdot V_k\cdot Z_{k+1}\cdot\big(\RLP_{n<k}V_n\big)\cdot\widetilde U_{k}\cdot \big(\RLP_{n<k}V_n\big)\cdot  Z_{k+1}\cdot V_k\cdot  \big(\LP_{k<n<m}Z_{n+1}V_n\big)=\\
&\big(\RP_{k\le n<m}V_nZ_{n+1}\big)\cdot \big(\RLP_{n<k}V_n\big)\cdot\widetilde U_{k}\cdot \big(\RLP_{n<k}V_n\big)\cdot  \big(\LP_{k\le n<m}Z_{n+1}V_n\big)=\Pi_{k}.
\end{aligned}
$$
Now we see that $$
\begin{aligned}
\big(\RLP_{n<m}V_n\big)^2&\subset \big(\RLP_{n<m}V_n\big)\cdot\widetilde U_m\cdot\big(\RLP_{n<m}V_n\big)=\Pi_{m}\subset\Pi_0=
\big(\RP_{n<m}V_nZ_{n+1}\big)\cdot\widetilde U_0\cdot \big(\LP_{n<m}Z_{n+1}V_n\big)\\
&\subset
\big(\RP_{n<m}U_{n+1}^2\big)\cdot \widetilde U_0\cdot \big(\LP_{n<m}U_{n+1}^2\big)\subset
\big(\RP_{n<m}W_{n+1}\big)\cdot W_0^2\cdot \big(\LP_{n<m}W_{n+1}\big)=\RLP_{n\le m}W_m.
\end{aligned}
$$
\end{proof}

\section{The independence of $\PTA$ and the balanced property}\label{PTA-SIN}

Looking at Theorems~\ref{t3.2} and \ref{t4.2} (which have the same conclusion) the reader can ask about the interplay between $\PTA$ and the balanced property. These two properties are independent.

First we present an example of a tower of topological groups $(G_n)_{n\in\w}$ that is balanced but does not satisfy $\PTA$.

Let $G=\HH_c(\IR)$ be the group of all homeomorphisms $h:\IR\to\IR$ having compact support $\supp(h)=\cl_\IR\{x\in\IR\rozd h(x)\ne x\}.$

The homeomorphism group $G=\HH_c(\IR)$ is endowed with the Whitney topology whose base at a homeomorphism $h\in \HH_c(\IR)$ consists of the sets 
$$B(h,\e)=\{f\in \HH_c(\IR)\rozd |f-h|<\e\}$$where $\e:\IR\to(0,1)$ runs over continuous positive functions on the real line.

It is well-known that the Whitney topology turns the homeomorphism group $G=\HH_c(\IR)$ into a topological group, see e.g., \cite{BMSY}. This group can be written as the countable union $G=\bigcup_{n\in\w}G_n$ of the closed subgroups
$$G_n=\big\{h\in G\rozd \supp(h)\subset[-n,n]\big\}.$$  Each subgroup $G_n$ can be identified with the group $\HH_+(\II_n)$ of orientation-preserving homeomorphisms of the closed interval $\II_n=[-n,n]$. The Whitney topology of the group $G$ induces on each subgroup $G_n$ the compact-open topology, generated by the sup-metric 
$\|f-h\|=\sup_{x\in\IR}|f(x)-h(x)|$.

In the following theorem we shall show that the topology of the direct limit $\glim G_n$ coincides with the topology  $\LA[\tau]$ on $G$ but the tower $(G_n)_{n\in\w}$ does not satisfy $\PTA$. This answers Problem~17.3 \cite{Glo06} of H.~Gl\"ockner.

\begin{theorem}\label{t6.1} \begin{enumerate}
\item The tower of the homeomorphism groups $\big(\HH_+(\II_n)\big)_{n\in\w}$ is balanced\textup{;}
\item the tower $\big(\HH_+(\II_n)\big)_{n\in\w}$ does not satisfy $\PTA$\textup{;}
\item The Whitney topology on $\HH_c(\IR)$ coincides with the topologies $\LA[\tau]$, $\RA[\tau]$, $\LR[\tau]$, $\RL[\tau]$ and those topologies coincide with the topology of the direct limit $\glim \HH_+(\II_n)$.
\end{enumerate}
\end{theorem}
 
\begin{proof} Let $G=\HH_c(\II)$ and $G_n=\HH_+(\II_n)\subset G$ for $n\in\w$. For a constant $\e>0$ the $\e$-ball $B(\id_\IR,\e)=\{f\in\HH_c(\IR)\rozd \|f-\id\|<\e\}\subset G$ centered at the identity homeomorphism $\id_\IR$ will be denoted by $B(\e)$. 
\smallskip

1. We need to show that for every $n\in\w$ the triple $(G_{n+2},G_{n+1},G_n)$ is balanced. This will follow as soon as we check that for every neighborhood $U\subset G_{n+1}$ of the identity homeomorphism $\id_\IR$ and any neighborhood $W\subset G_{n+2}$ of $\id_\IR$ the set $U\cdot\sqrt[G_n]{W}$ is a neighborhood of $\id_\IR$ in $G_{n+2}$. Since the Whitney topology on the subgroup $G_{n+2}=\HH_+(\II_{n+2})$ is generated by the sup-metric,  the neighborhood $W\subset G_{n+2}$ contains the $\e$-ball $G_{n+2}\cap B(\e)$ for some positive constant $\e<1$. The constant $\e$ can be chosen so small that $G_{n+1}\cap B(\e)\subset U$.

Consider the closed subgroup $$H=\{h\in G_{n+2}\rozd \supp(h)\subset\II_{n+2}\setminus\II_{n}\}$$ of $G_{n+2}$ and observe that $W\cap H\subset \sqrt[G_n]{W}$. Now it suffices to check that $U\cdot (W\cap H)$ contains the ball $G_{n+2}\cap B(\e/2)$. Take any homeomorphism $h\in G_{n+2}\cap B(\e/2)$ and observe that $h$ maps the interval 
$\II_n=[-n,n]$ into the interval $[-n-\e/2,n+\e/2]$. So, we can consider the homeomorphism $g\in G_{n+1}$, which is equal to $h$ on the interval $\II_n$ and is linear on the intervals $[n,n+1]$ and $[-n-1,-n]$. It is clear that $\|g-\id\|\le\|h-\id\|<\e/2$ and $\|g^{-1}-\id\|=\|g-\id\|$. Let $f=h\circ g^{-1}\in G_{n+2}$. The equality $g|\II_n=h|\II_n$ implies $f|\II_n=\id|\II_n$ and thus $f\in H$. It follows that $\|f-\id\|=\|h\circ g^{-1}-\id\|\le\|h\circ g^{-1}-g^{-1}\|+\|g^{-1}-\id\|<\e/2+\e/2=\e$. Now we see that the elements $g\in G_{n+1}\cap B(\e/2)\subset U$ and $f=g^{-1}\circ h\in H\cap B(\e)\subset \sqrt[G_n]{W}$ yield $h=g\circ f\subset U\cdot\sqrt[G_n]{W}$, which establishes the required inclusion $G_{n+2}\cap B(\e/2)\subset U\cdot\sqrt[G_n]{W}$.
\smallskip

2. Assuming that the tower $(G_n)_{n\in\w}$ satisfies $\PTA$, we can find a neighborhood $U\subset G_1$ of $\id_\IR$ such that for every neighborhood $V\subset G_2$ there is a neighborhood $W\subset G_2$ such that $WU\subset UV$.

Find $\e\in(0,1)$ such that $U\supset G_1\cap B(\e)$. Then for the neighborhood $V=G_2\cap B(\e)$ there is a neighborhood $W\subset G_2$ of $\id_\IR$ with $WU\subset UV$. Find a positive constant $\delta<\e$ with $G_2\cap B(\delta)\subset W$.
It follows that $(G_2\cap B(\delta))\cdot (G_1\cap B(\e))\subset WU\subset UV\subset G_1\cdot (G_2\cap B(\e))$ and after inversion, $(G_1\cap B(\e))\cdot(G_2\cap B(\delta))\subset (G_2\cap B(\e))\cdot G_1$. Take a homeomorphisms $f\in G_2\cap B(\delta)$ such that $f(1)=1-\delta/2$ and a homeomorphism $g\in G_1\cap B(\e)$ such that $g(1-\delta/2)=1-\e$. Then $g\circ f(1)=1-\e$ which is not possible as $g\circ f\in (G_2\cap B(\e))\cdot G_1\subset\{h\in G_2:|h(1)-1|<\e\}$.
\smallskip

3. Since the tower $(G_n)_{n\in\w}$ is balanced, the topology of $\glim G_n$ coincides with the topologies $\LA[\tau]$, $\RA[\tau]$, $\LR[\tau]$, and $\RL[\tau]$ according to Theorem~\ref{t4.2}. Since $\glim G_n$ carries the strongest group topology inducing the original topology on each group $G_n$, we conclude that the Whitney topology is weaker that the topology $\LA[\tau]$. In order to show that these two topologies coincide, it suffices to check that each basic neighborhood $\LP_{n\in\w}U_n$ of $e$ in the topology $\LA[\tau]$ is a neighborhood of $\id_\IR$ in the Whitney topology. Here for every $n\in\w$, $U_n$ is an open symmetric neighborhood of $\id_\IR$ in the group $G_n=\HH_+(\II_n)$. Since the Whitney topology on the subgroup $G_n$ is generated by the sup-metric, we can find a positive constant $\e_n<1/2$ such that $U_n\supset G_n\cap B(\e_n)$.

Choose a continuous function $\e:\IR\to(0,\tfrac12)$ such that 
\begin{equation}\label{E1}
\sup\{\e(x)\rozd x\in\II_{n}\setminus \II_{n-4}\}<\e_n/2\mbox{ \ for all $n\in\w$}.
\end{equation} 
Here we assume that $\II_k=\emptyset$ for all negative $k$.

The function $\e$ determines the neighborhood $B(\e)=\{h\in G\rozd  |h-\id_\IR|<\e\}$ of the identity map $\id_\IR$ in the Whitney topology. 

We claim that $B(\e)\subset\LP_{n\in\w}U_n$. 
Fix any homeomorphism $h\in B(\e)$ and for every $n\in\IN$ consider the homeomorphism $h_n\in G_{n}$ such that $h_n|\II_{n-1}=h|\II_{n-1}$ and $h_n$ is linear on the intervals $[n,n+1]$ and $[-n-1,-n]$. For $n\le 0$ we put $h_n=\id_\IR$.
It is clear that $h_{m}=h$ for some $m\in\IN$.

For every $n\in\w$ consider the homeomorphism $g_{n}=h_{n-1}^{-1}\circ h_{n}\in G_n$. Then $h=h_{m}=\LP_{n\le m}g_n$. It remains to prove that each homeomorphism $g_n$ belongs to the neighborhood $U_n$. This will follow as soon as we check that $|g_n(x)-x|<\e_n$ for any $x\ne g_n(x)$.

Since $h_{n-1}|\II_{n-2}=h_{n}|\II_{n-2}=h|\II_{n-2}$, we conclude that $x\in \II_{n}\setminus\II_{n-2}$. 
It follows from $h_{n}\in G_n\cap B(1/2)$ that the point $y=h_{n}(x)$ belongs to the set $\II_{n}\setminus\II_{n-3}$. Since $h_{n-1}\in G_{n-1}\cap B(1/2)$, the point $z=h_{n-1}^{-1}(y)$ belongs to $\II_{n}\setminus\II_{n-4}$. 

We claim that 
\begin{equation}\label{H1}
|h_{n-1}(z)-z|<\e_{n}/2.
\end{equation} If $z\in \II_{n-2}\setminus\II_{n-4}$, then 
$|h_{n-1}(z)-z|=|h(z)-z|<\e(z)\le \e_{n}/2$ by the condition (\ref{E1}) from the definition of the function $\e$. If $z\in \II_{n-1}\setminus\II_{n-2}$, then the linearity of $h_{n-1}$ on the two intervals composing the set $\II_{n-1}\setminus\II_{n-2}$ implies that $$|h_{n-1}(z)-z|\le\max_{t\in\partial\II_{n-2}}|h(t)-t|<\max_{t\in\partial\II_{n-2}}\e(t)\le
\e_{n}/2$$by the definition of the function $\e$. Here $\partial\II_{k}=\{k,-k\}$ stands for the boundary of the interval $\II_{k}=[-k,k]$ in $\IR$. 

By a similar argument we can prove the inequality
\begin{equation}\label{H2}
|h_{n}(x)-x|<\e_{n}/2.
\end{equation}

Unifying (\ref{H1}) and (\ref{H2}) we obtain the desired inequality:
$$
\begin{aligned}
|g_{n}(x)-x|&=|h_{n-1}^{-1}\circ h_{n}(x)-x|\le |h_{n-1}^{-1}\circ h_{n}(x)-h_{n}(x)|+|h_{n}(x)-x|=\\
&=|z-h_{n-1}(z)|+|h_{n}(x)-x|<\frac12\e_{n}+\frac12\e_n=\e_n.
\end{aligned}
$$
\end{proof}

%\begin{remark} In \cite{BMSY} Theorem~\ref{t6.1} will be generalized to groups of compactly supported diffeomorphisms of $C^r$ manifolds.
%\end{remark}

Next, we present an example of a tower $(G_n)_{n\in\w}$ that satisfies $\PTA$ but is not (bi-) balanced. 

\begin{example} Let $(e_n)_{n\in\w}$ be an orthonormal basis of the separable Hilbert space $l_2$ and $\mathcal B(l_2)$ be the Banach algebra of bounded linear operators on $l_2$. For every $n\in\IN$ let $G_n$ be the subgroup of $\mathcal B(l_2)$ consisting of invertible linear operators $T:l_2\to l_2$ such that 
\begin{itemize}
\item $Te_0\in (0,+\infty)\cdot e_0$;
\item $Te_i\in e_i+\IR\cdot e_0$ for all $1\le i\le n$;
\item $Te_i=e_i$ for all $i>n$.
\end{itemize}

The tower $(G_n)_{n\in\IN}$ satisfies $\PTA$ because each group $G_n$ is locally compact, see \cite{TSH}, \cite{HSTH}.
On the other hand, for every $n\in\IN$ the triple $(G_n,G_{n+1},G_{n+2})$ is not bi-balanced. The reason is that for the neighborhood $W=\{T\in G_{n+2}:Te_{n+2}\in e_{n+2}+(-1,1)e_0\}$ of the identity $\id$ in $G_{n+2}$ the set $Z=\sqrt[G_n]{W}$ lies in the subgroup $G_{n+1}$. Then for each neighborhood $V\subset G_{n+1}$, the product $ZVZ\subset G_{n+1}$ fails to be a neighborhood of $\id$ in the group $G_{n+2}$.
\end{example}

It is clear that each balanced triple of groups is bi-balanced. The converse implication is not true.

\begin{example} In the group $G=GL(3,\IR)$ of non-degenerated $3\times 3$-matrices consider the subgroups
$$
\Gamma=\left\{\left(\begin{array}{ccc}
a_{11}&a_{12}&0\\
a_{21}&a_{22}&0\\
0&0&a_{33}\end{array}\right)\in G\right\}\mbox{ \ \ and \ \ }
H=\left\{\left(\begin{array}{ccc}
a_{11}&0&0\\
0&1&0\\
0&0&1\end{array}\right)\in G\right\}.
$$
It is easy to check that the triple $(G,\Gamma,H)$ is bi-balanced but not balanced.
\end{example}

\section{Direct limits in the category of uniform spaces}\label{s:ulim}

In this section we shall discuss the notion of the direct limit in the category of uniform spaces and their uniformly continuous maps. In Section~\ref{s:ulim-top} we shall apply those results to show that for a tower $(G_n)_{n\in\w}$ of topological groups the topologies $\LA[\tau]$ and $\RA[\tau]$ on the union $G=\bigcup_{n\in\w}G_n$ are generated by uniformities of direct limits of the groups $G_n$ endowed with the left and right uniformities. 

Fundamenta of the theory of uniform spaces can be found in \cite[Ch.8]{En}. Uniformities on groups are thoroughly discussed in \cite{RD} and \cite[\S1.8]{AT}. In the sequel, for a uniform space $X$ by $\U_X$ we shall denote the uniformity of $X$.

Let $$X_0\to X_1\to X_2\to\cdots
$$be a sequence of uniform spaces and their injective uniformly continuous maps. We shall identify each space $X_n$ with a subset of the uniform space $X_{n+1}$, carrying its own uniformity, which is stronger than that inherited from $X_{n+1}$. By the uniform direct limit $\ulim X_n$ of the sequence of uniform spaces $(X_n)_{n\in\w}$ we understand the union $X=\bigcup_{n\in\w}X_n$ endowed with the strongest (not necessarily separated) uniformity turning the identity inclusions $X_n\to X$, $n\in\w$, into uniformly continuous maps. 

A sequence
$$X_0\to X_1\to X_2\to\cdots$$ of uniform spaces is called a {\em tower} of uniform spaces if each uniform space $X_n$ is a subspace of the uniform space $X_{n+1}$, so  the identity inclusion $X_n\to X_{n+1}$ is a uniform embedding. 

The uniformity of the uniform direct limit $\ulim X_n$ of a tower $(X_n)_{n\in\w}$ of uniform spaces was described in \cite{BaR} with help of uniform pseudometrics. 

Let us recall that a pseudometric on a uniform space $Y$ is {\em uniform} if for every $\e>0$ the set $$\{d<\e\}:=\{(x,y)\in Y\rozd d(x,y)<\e\}$$ belongs to the uniformity $\U_Y$ of $Y$. By \cite[8.1.10]{En}, the uniformity $\U_Y$ of a uniform space $Y$ is generated by the family $\PM_Y$ of all uniform pseudometrics on $Y$ in the sense that the sets $\{d<1\}$, $d\in\PM_Y$, form a base of the uniformity $\U_Y$. 

Let $(X_n)_{n\in\w}$ be a tower of uniform spaces. A sequence of pseudometric $(d_n)_{n\in\w}\in\tprod_{n\in\w}\PM_{X_n}$
is called {\em monotone} if $d_n\le d_{n+1}|X_n^2$ for every $n\in\w$. Let
$$\mprod_{n\in\w}\PM_{X_n}=\{(d_n)_{n\in\w}\in\tprod_{n\in\w}\PM_{X_n}\rozd (d_n)_{n\in\w}\mbox{ is monotone}\}$$be the subspace of Cartesian product, consisting of monotone sequences of uniform pseudometrics on the uniform spaces $X_n$.

A family $\A\subset\mprod\limits_{n\in\w}\PM_{X_n}$ is defined to be {\em adequate} if 
for each sequence of entourages $(U_n)_{n\in\w}\in\prod_{n\in\w}\U_{X_n}$ there is a monotone sequence of uniform pseudometrics $(d_n)_{n\in\w}\in\A$ such that $\{d_n<1\}\subset U_n$ for all $n\in\w$.

The following proposition proved in \cite{BaR} shows that adequate families exist.

\begin{proposition} For any tower $(X_n)_{n\in\w}$ of uniform spaces the family $\A=\mprod\limits_{n\in\w}\PM_{X_n}$ is adequate.
\end{proposition}

For a point $x\in X=\bigcup_{n\in\w}X_n$ let $|x|=\min\{n\in\w\rozd x\in X_n\}$ be the {\em height} of $x$ in $X$. For two points $x,y\in X$ put $|x,y|=\max\{|x|,|y|\}$. Now we define a limit operator $\dlim$ assigning to each sequence of pseudometrics $(d_n)_{n\in\w}\in\prod_{n\in\w}\PM_{X_n}$ the pseudometric $d_\infty=\dlim d_n$ on $X$ defined by the formula
$$d_\infty(x,y)=\inf\Big\{\sum_{i=1}^md_{|x_{i-1},x_i|}(x_{i-1},x_i)\rozd x=x_0,x_1,\dots,x_n=y\Big\}.$$
In fact, the pseudometric $\dlim d_n$ is well-defined for any functions $d_n:X_n\times X_n\to [0,\infty)$, $n\in\w$, such that $d_n(x,x)=0$ and $d_n(x,y)=d_n(y,x)$ for all $x,y\in X_n$.

The following theorem proved in \cite{BaR} describes the uniformity of uniform direct limits.

\begin{theorem}\label{t7.2} For a tower of uniform spaces $(X_n)_{n\in\w}$ and an adequate family $\A\subset\mprod\limits_{n\in\w}\PM_{X_n}$ the uniformity of the uniform direct limit $\ulim X_n$ is generated by the family of pseudometrics $\{\dlim d_n\rozd (d_n)_{n\in\w}\in\A\}.$
\end{theorem}

Theorem~\ref{t7.2} implies a simple description of the topology of the uniform limit $\ulim X_n$ also given in \cite{BaR}. Given two subsets $U,V\subset X^2$ of the square of $X$, consider their composition (as relations):
$$A\circ B=\{(x,z)\in X^2\rozd \mbox{ there is $y\in X$ such that $(x,y)\in A$ and $(y,z)\in B$ }\}.$$
This operation can be extended to finite and infinite sequences of subsets $(A_n)_{n\in\w}$ of $X^2$ by the formula 
$$\sum_{n\ge k}A_n=\bigcup_{n\ge k}A_k\circ A_{k+1}\cdots \circ A_n.$$

For a point $x$ of a set $X$ and a subset $U\subset X^2$ let $B(x,U)=\{y\in X\rozd (x,y)\in U\}$ be the $U$-ball centered at $x$. We recall that for a point $x$ of the union $X=\bigcup_{n\in\w}X_n$ of a tower $(X_n)$ by $|x|=\min\{n\in\w\rozd x\in X_n\}$ we denote the height of $x$ in $X$.

\begin{theorem}\label{t7.3} The topology of the uniform direct limit $\ulim X_n$ of a tower of uniform spaces $(X_n)$ is generated by the base
$$\mathcal B=\Big\{B\big(x; \sum_{n\ge|x|}U_n\big)\rozd x\in X,\;\;(U_n)_{n\ge|x|}\in\prod\limits_{n\ge|x|}\U_{X_n}\Big\}.$$
\end{theorem}

\section{Uniformities on groups}\label{s:ug}

In this section we discuss some natural uniformities on topological groups. For more information on this subject, see \cite{RD} and \cite[\S1.8]{AT}.

Let us recall that each topological group $G$ carries four natural uniformities:
\begin{itemize}
\item[1)] the {\em left uniformity\/} $\U^{\Ls}$, generated by the entourages $U^{\Ls}=\{(x,y)\in G\rozd x\in yU\}$ where $U\in\BB_e$;
\item[2)] the {\em right uniformity\/} $\U^{\Rs}$, generated by the entourages $U^{\Rs}=\{(x,y)\in G\rozd x\in Uy\}$ where $U\in\BB_e$;
\item[3)] the {\em two-sided uniformity\/} $\U^{\LRs}$, generated by the entourages  $U^{\LRs}=\{(x,y)\in G{\rozd}\allowbreak x\in yU\cap Uy\}$ with $U\in\BB_e$;
\item[4)] the {\em Roelcke uniformity\/} $\U^{\RLs}$, generated by the entourages $U^{\RLs}=\{(x,y)\in G\rozd\allowbreak x\in UyU\}$ with $U\in\BB_e$.
\end{itemize}
Here $\BB_e$ stands for the family of open symmetric neighborhoods $U=U^{-1}\subset G$ of the neutral element $e$ in the topological group $G$.

The group $G$ endowed with the uniformity $\U^{\Ls}$, $\U^{\Rs}$, $\U^{\LRs}$ or $\U^{\RLs}$ will be denoted by $G^{\Ls}$, $G^{\Rs}$, $G^{\LRs}$ or $G^{\RLs}$, respectively. It follows from the definition of those uniformities that the identity maps in the following diagram are uniformly continuous:

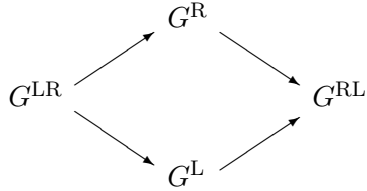
\begin{picture}(200,90)(-130,-10)
\put(5,30){$G^{\LRs}$}
\put(30,28){\vector(3,-2){30}}
\put(30,38){\vector(3,2){30}}
\put(65,0){$G^{\Ls}$}
\put(85,6){\vector(3,2){30}}
\put(65,60){$G^{\Rs}$}
\put(85,58){\vector(3,-2){30}}
\put(120,30){$G^{\RLs}$}
\end{picture}

Any isomorphic topological embedding $H\hookrightarrow G$ of topological groups induces uniform embeddings 
$$H^{\Ls}\hookrightarrow G^{\Ls},\;\;H^{\Rs}\hookrightarrow G^{\Rs}, \;\; H^{\LRs}\hookrightarrow G^{\LRs}$$ of the corresponding uniform spaces, see Proposition~1.8.4 of \cite{AT}. For the Roelcke uniformity the induced map 
$H^{\RLs}\to G^{\RLs}$ is merely uniformly continuous, but is not necessarily a uniform embedding, see \cite{Us}.

Let us observe that $G^{\LRs}$, $G^{\Ls}$, $G^{\Rs}$, $G^{\RLs}$ are groups endowed with uniformities which are tightly connected with their algebraic structure.

By analogy with semitopological and quasitopological groups, let us define a group $G$ endowed with a uniformity to be a
\begin{itemize}
\item {\em semiuniform group} if left and right shifts on $G$ are uniformly continuous;
\item {\em quasiuniform group} if $G$ is a semiuniform group with uniformly continuous inversion;
\item {\em uniform group} if $G$ is a quasiuniform group with uniformly continuous multiplication $G\times G\to G$, $(x,y)\mapsto xy$.
\end{itemize}

The groups $G^{\Ls}$, $G^{\Rs}$, $G^{\LRs}$, $G^{\RLs}$ are basic examples of groups endowed with a uniformity.
Some elementary properties of those groups are presented in the following two propositions whose proof is left to the interested reader (cf. Corollary 1.8.16 \cite{AT}).

\begin{proposition} For any topological group $G$
\begin{enumerate}
\item $G^{\Ls}$ and $G^{\Rs}$ are semiuniform topological groups\textup{;}
\item $G^{\LRs}$ and $G^{\RLs}$ are quasiuniform topological groups.
\end{enumerate}
\end{proposition}

%\begin{proof} To show that $G^{\Ls}$ is a semiuniform group we need to show that for every $a\in G$ the left shift
%$$l_a:G^{\Ls}\to G^{\Ls},\;\;l_a:x\mapsto ax$$
%and the right shift
%$$r_a:G^{\Ls}\to G^{\Ls},\;\;r_a:x\mapsto ax$$
%are uniformly continuous. Take any entourage $U^{\Ls}\in\U^{\Ls}$.
%We recall that $U^{\Ls}=\{(x,y)\in G^2:x\in yU\}$ for some open symmetric neighborhood $U=U^{-1}$ of the neutral element $e$ in the topological group $G$.

%Observe that for every $(x,y)\in U^{\Ls}$ we get $(l_a(x),l_a(y))=(ax,ay)\in U^{\Ls}$, so the left shift $l_a:G^{\Ls}\to G^{\Ls}$ is uniformly continuous. To prove the uniform continuity of the right shift $r_a$, find a symmetric neighborhood $V\subset G$ of $e$ such that $Va\subset aU$ and observe that for every $(x,y)\in V^{\Ls}$ we get $y\in xV$ and thus 
%$ya\in xVa\subset xaU$, which means $(r_a(x),r_a(y))=(xa,ya)\in U^{\Ls}$ and hence the right shoft $r_a:G^{\Ls}\to G^{\Ls}$ is uniformly continuous. 

%By analogy we can prove that $G^{\Rs}$, $G^{\LRs}$, and $G^{\RLs}$ are semiuniform groups. The semiuniform group $G^{\LRs}$ is a quasiuniform group becuase for every open symmetric neighborhood $U\subset G$ of $e$ a pair $(x,y)$ belongs to the entourage $U^{\LRs}$ if and only if the pair $(x^{-1},y^{-1})$ belongs to $U^{\LRs}$. The same is true for the entourage $U^{\RLs}$.
%\end{proof}

\begin{proposition} For a topological group $G$ the following conditions are equivalent:
\begin{enumerate}
\item $G^{\Ls}$ is a quasiuniform group\textup{;}
\item $G^{\Ls}$ is a uniform group\textup{;}
\item $G^{\Rs}$ is a quasiuniform group\textup{;}
\item $G^{\Rs}$ is a uniform group\textup{;}
\item $G^{\LRs}$ is a uniform group\textup{;}
\item $G^{\RLs}$ is a uniform group\textup{;}
\item the left and right uniformities on $G$ coincide\textup{;}
\item the group $G$ is balanced.
\end{enumerate}
\end{proposition}

\section{The uniform structure of the semitopological groups $\LA[G]$, $\RA[G]$, and $\LR[G]$}\label{s:ulim-top}

Each tower of topological groups $(G_n)_{n\in\w}$ induces four ascending sequences of uniform spaces
$(G^{\Ls}_n)_{n\in\w}$, $(G_n^{\Rs})_{n\in\w}$, $(G_n^{\LRs})_{n\in\w}$, $(G_n^{\RLs})_{n\in\w}$. The direct limits of these sequences in the category of uniform spaces are denoted by
$$\ulim G_n^{\Ls},\quad\ulim G_n^{\Rs},\quad\ulim G_n^{\LRs},\mbox{ \ and \ }\ulim G_n^{\RLs},$$respectively.

These uniform spaces endowed with the group operation inherited from $G=\bigcup_{n\in\w}G_n$ are semiuniform groups. The uniform continuity of the left and right shifts follows from the uniform continuity of the left and right shifts on the semiuniform groups $G_n^{\Ls}$, $G_n^{\Rs}$, $G_n^{\LRs}$, $G_n^{\RLs}$, $n\in\w$. Moreover, the semiuniform groups $\ulim G_n^{\LRs}$ and $\ulim G_n^{\RLs}$ are quasiuniform because so are the groups $G_n^{\LRs}$ and $G_n^{\RLs}$, $n\in\w$.

The uniform continuity of the identity maps

\begin{picture}(200,80)(-100,-2)
\put(5,30){$G_n^{\LRs}$}
\put(30,28){\vector(3,-2){30}}
\put(30,38){\vector(3,2){30}}
\put(65,0){$G_n^{\Ls}$}
\put(85,6){\vector(3,2){30}}
\put(65,60){$G_n^{\Rs}$}
\put(85,58){\vector(3,-2){30}}
\put(120,30){$G_n^{\RLs}$}
\put(145,32){\vector(1,0){30}}
\put(185,30){$(\glim G_n)^{\RLs}$}
\end{picture}

for all $n\in\w$ implies the uniform continuity of the identity maps:

\begin{picture}(200,90)(-100,-10)
\put(-10,30){$\ulim G_n^{\LRs}$}
\put(25,23){\vector(2,-1){25}}
\put(25,42){\vector(2,1){25}}
\put(50,0){$\ulim G_n^{\Ls}$}
\put(80,12){\vector(2,1){25}}
\put(50,60){$\ulim G_n^{\Rs}$}
\put(80,53){\vector(2,-1){25}}
\put(100,30){$\ulim G_n^{\RLs}$}
\put(150,32){\vector(1,0){30}}
\put(190,30){$(\glim G_n)^{\RLs}$}
\end{picture}

Theorems~\ref{t7.2} and \ref{t7.3} imply the following description of the uniform and topological structure of the uniform limit $\ulim G_n^{\Ls}$. 

\begin{theorem}For a tower of topological groups $(G_n)_{n\in\w}$
\begin{enumerate}
\item the topology of the semiuniform group $\ulim G_n^{\Ls}$ coincides with the topology $\LA[\tau]$ on the group  $G=\bigcup_{n\in\w}G_n$\textup{;}
\item the uniformity of $\ulim G_n^{\Ls}$ is generated by the family of pseudometrics 
$\{\dlim d_n\rozd (d_n)_{n\in\w}\in\A\}$ for any adequate family $\A\subset\mprod\limits_{n\in\w}\PM_{G_n}$.
\end{enumerate}
\end{theorem}

This theorem allows us to identify the semitopological group $\LA[G]$ with the semiuniform group $\ulim G_n^{\Ls}$.
In the same way we shall identify the semitopological group $\RA[G]$ with the semiuniform group $\ulim G_n^{\Rs}$.

The semitopological group $\LR[G]$ is a quasiuniform group with respect to the uniformity inherited from the product $\RA[G]\times\LA[G]$ by the diagonal embedding $$\LR[G]\hookrightarrow \RA[G]\times\LA[G],\quad x\mapsto (x,x).$$

The uniform continuity of the identity maps 
$$\ulim G_n^{\LRs}\to\ulim G_n^{\Ls}=\LA[G]\mbox{ \  and \ }
\ulim G_n^{\LRs}\to\ulim G_n^{\Rs}=\RA[G]$$ yields the uniform continuity of the identity map $\ulim G_n^{\LRs}\to\LR[G]$.

Now we discuss the interplay between the semitopological group $\RL[G]$ and the semiuniform group $\ulim G_n^{\RLs}$.
Since the topological embedding $G_n^{\RLs}\to G_{n+1}^{\RLs}$ in general is not a uniform embedding, Theorems~\ref{t7.2} and \ref{t7.3} cannot be applied to describing the uniform and topological structures of the uniform direct limit $\ulim G_n^{\RLs}$. So, this case requires a special treatment.

Given a pseudometric $d$ on a group $H$, let $d^{-1}$ be the {\em mirror pseudometric}  defined by
$$d^{-1}(x,y)=d(x^{-1},y^{-1})\mbox{ \ for $x,y\in H$}.$$

\begin{theorem} For a tower of topological groups $(G_n)_{n\in\w}$,
\begin{enumerate}
\item the uniformity of the uniform direct limit $\ulim G_n^{\RLs}$ is generated by the family of pseudometrics $\big\{\dlim \min\{d_n,d^{-1}_n\}\rozd (d_n)_{n\in\w}\in\mprod_{n\in\w}\PM_{G_n^{\Ls}}\big\};$
\item the uniformity of $\ulim G_n^{\RLs}$ coincides with the strongest uniformity on the group $G=\bigcup_{n\in\w}G_n$  such that the identity maps $\LA[G]\to G$  and $\RA[G]\to G$ are uniformly continuous\textup{;}
\item the identity map $\RL[G]\to\ulim G_n^{\RLs}$ is continuous\textup{;}
\item the identity map $\RL[G]\to\ulim G_n^{\RLs}$ is a homeomorphism if $\RL[G]$ is a topological group or if each identity inclusion $G_n^{\RLs}\to G_{n+1}^{\RLs}$, $n\in\w$, is a uniform embedding\textup{;}
\item $\ulim G_n^{\RLs}$ is a topological group if and only if the identity map $\ulim G_n^{\RLs}\to\glim G_n$ is a homeomorphism.
\end{enumerate}
\end{theorem}

\begin{proof} 1. First we show that for any monotone sequence of pseudometrics $(d_n)_{n\in\w}\in\mprod\limits_{n\in\w}\PM_{G_n^{\Ls}}$ the pseudometric $d_\infty=\dlim \min\{d_n,d_n^{-1}\}$ is uniform on $G=\ulim G_n^{\RLs}$. For this it suffices to check that $d_\infty$ is uniform on each quasiuniform group $G_n^{\RLs}$, $n\in\w$.

 Fix any $\e>0$. Since each  pseudometric $d_n\in\PM_{G_n^{\Ls}}$ is uniform with respect to the left uniformity on the topological group $G_n$, there is an open symmetric neighborhood $U_n\subset G_n$ of $e$ such that 
$U_n^{\Ls}\subset\{d_n<\e/2\}$. After inversion, we get the inclusion  $U_n^{\Rs}\subset\{d_n^{-1}<\e/2\}$. We claim that $U_n^{\RLs}\subset\{d_\infty|G_n^2<\e\}$. Take any points $(x,y)\in U_n^{\RLs}$. Then $y=uxv$ for some $u,v\in U_n$. 
Consider the chain of points $x_0=x$, $x_1=ux$, $x_2=uxv=y$ and observe that 
$$
\begin{aligned}
d_\infty(x,y)&\le \min\{d_{|x_0,x_1|}(x_0,x_1),d^{-1}_{|x_0,x_1|}(x_0,x_1)\}+
\min\{d_{|x_1,x_2|}(x_1,x_2),d^{-1}_{|x_1,x_2|}(x_1,x_2)\}\\
&\le \min\{d_n(x,ux),d^{-1}_n(x,ux)\}+\min\{d_n(ux,uxv),d^{-1}_n(ux,uxv)\}\\   &\le d_n(x^{-1},x^{-1}u^{-1})+d_n(ux,uxv)<\frac\e2+\frac\e2=\e,
\end{aligned}$$
witnessing that the pseudometric $d_\infty$ is uniform.
  
Now given any entourage $U\subset G^2$ that belongs to the uniformity of the space $\ulim G_n^{\RLs}$, we shall find a monotone sequence $(d_n)_{n\in\w}\in\mprod_{n\in\w}\PM_{G_n^{\Ls}}$ such that 
$\{d_\infty<1\}\subset U$ for the limit
pseudometric $d_\infty=\dlim\min\{d_n,d_n^{-1}\}$. 

By \cite[8.1.10]{En}, there is a uniform pseudometric $d$ on $\ulim G_n^{\RLs}$ such that  $\{(d<1\}\subset U$. Since the inversion is uniformly continuous on $\ulim G_n^{\RLs}$, the pseudometric $\rho=\max\{d,d^{-1}\}$ on $\ulim G_n^{\RLs}$ is uniform.
Now observe that for every $n\in\w$ the pseudometric $d_n=\rho|(G_n^{\RLs})^2$ on $G_n^{\RLs}$ is uniform and the sequence $(d_n)_{n\in\w}$ is monotone. The triangle inequality and the definition of the pseudometric $d_\infty=\dlim \min\{d_n,d_n^{-1}\}=\dlim d_n$ implies that $d_\infty=\rho$. In this case 
$\{d_\infty<1\}=\{\rho<1\}\subset\{d<1\}\subset U.$
\smallskip

2. Let $\U$ be the largest uniformity on $G$ such that the identity maps $\LA[G]\to(G,\U)$ and $\RA[G]\to(G,\U)$ are uniformly continuous. The definition of $\U$ implies that $\U$ coincides with its mirror uniformity $\U^-$ consisting of the sets $U^-=\{(x^{-1},y^{-1})\rozd (x,y)\in U\}$, $U\in\U$.

We need to show that $\U$ coincides with the uniformity of $\ulim G_n^{\RLs}$. The uniform continuity of the identity maps from $\LA[G]$ and $\RA[G]$ into $\ulim G_n^{\RLs}$ implies that $\U$ is larger then the uniformity of $\ulim G_n^{\RLs}$.
It remains to prove that each entourage $U\in\U$ belongs to the uniformity of $\ulim G_n^{\RLs}$. 
By \cite[8.1.10]{En}, there is a uniform pseudometric $d$ on $(G,\U)$ such that $\{d<1\}\subset U$. Since the inversion on $G$ is uniformly continuous with respect to the uniformity $\U$, the mirror pseudometric $d^{-1}$ is uniform on $(X,\U)$ and so is the pseudometric $\rho=\max\{d,d^{-1}\}$. Now we see that for each $n\in\w$ the restriction $d_n=\rho|G^2_n$ belongs to the family $\PM_{G_n^{\Ls}}$ and is equal to its mirror pseudometric $d_n^{-1}$. 
The sequence $(d_n)_{n\in\w}$ belongs to $\mprod\limits_{n\in\w}\PM_{G_n^{\Ls}}$, and the definition of the pseudometric $d_\infty$ implies that $d_\infty=\rho$. By the first item, the pseudometric $d_\infty$ is uniform on $\ulim G_n^{\LRs}$ and consequently, the entourage $$U\supset\{d<1\}\supset\{\rho<1\}=\{d_\infty<1\}$$belongs to the uniformity of the space $\ulim G_n^{\RLs}$. 
\smallskip

3. Since $\RL[G]$ and $\ulim G_n^{\RLs}$ are semitopological groups, the continuity of the identity map $\RL[G]\to\ulim G_n^{\RLs}$ is equivalent to its continuity at the neutral element $e$.

Given a neighborhood $O^{\RLs}(e)\subset\ulim G_n^{\RLs}$ of $e$, find a uniform pseudometric $d$ on $\ulim G_n^{\RLs}$ such that $\{x\in G\rozd d(x,e)<1\}\subset O^{\RLs}(e)$. For every $n\in\w$ the identity map $G_n^{\RLs}\to\ulim G_n^{\RLs}$ is uniformly continuous, so we can find a symmetric neighborhood $U_n\subset G_n$ of $e$ such that $U_n^{\RLs}\subset \big\{d|G_n^2<1/{2^{n+1}}\big\}$. By the definition of the topology $\RL[\tau]$ of the quasitopological group $\RL[G]$, the set $\RLP_{n\in\w}U_n$ is a neighborhood of $e$ in $\RL[G]$. We claim that 
$\RLP_{n\in\w}U_n\subset O^{\RLs}(e)$. Given any point $z\in \RLP_{n\in\w}U_n$, find two points $x\in\RP_{n\in\w}U_n$ and $y\in\LP_{n\in\w}U_n$ with $z=xy$.

By the definition of the directed products $\RP_{n\in\w}U_n$ and $\LP_{n\in\w}U_n$, there are chains of points $e=x_0,x_1,\dots, x_m=x$ and $e=y_0,y_1,\dots,y_m$ in $G$ such that $x_{i+1}\in U_ix_i$ and $y_{i+1}\in y_iU_i$ for all $i<m$. Now consider the chain $e=x_0y_0,x_1y_1,\dots,x_my_m=xy$ linking the points $e$ and $z=xy$.
Observe that for every $i<m$, \  
$x_{i+1}y_{i+1}\in U_ix_iy_{i}U_i$ implies $(x_{i+1}y_{i+1},x_iy_{i})\in U_i^{\RLs}$ and hence $d(x_{i+1}y_{i+1},x_iy_{i+1})<1/{2^{i+1}}$ by the choice of the neighborhood $U_i$. Consequently, 
$$d(e,xy)\le\sum_{i<m}d(x_iy_i,x_{i+1}y_{i+1})\le\sum_{i<m}\frac1{2^{i+1}}<1$$
and $xy\in O^{\RLs}(e)$ by the choice of the pseudometric $d$.
\smallskip

4a. If $\RL[G]$ is a topological group, then  the identity map $\RL[G]\to\glim G_n$ is a homeomorphism by Theorem~\ref{t2.2}. By the preceding item, the identity map $\RL[G]\to\ulim G_n^{\RLs}$ is continuous, and has continuous inverse, which is the composition of two continuous maps $\ulim G_n^{\RLs}\to\glim G_n\to \RL[G]$. 
\smallskip

4b. Assume that the identity maps $G_n^{\RLs}\to G_{n+1}^{\RLs}$, $n\in\w$, are uniform embeddings. In this case Theorem~\ref{t7.3} implies that the uniform direct limit $\ulim G_n^{\RLs}$ has the family $$\Big\{B\big(e;\sum_{n\in\w}U_n^{\RLs}\big)\rozd (U_n)_{n\in\w}\in\tprod_{n\in\w}\BB_n\Big\}$$as a neighborhood base at $e$. Since each set $B(e;\sum_{n\in\w}U_n^{\RLs})$ coincides with $\RLP_{n\in\w}U_n$, we see that the topologies of the semitopological groups $\ulim G_n^{\RLs}$ and $\RL[G]$ coincide at $e$ and thus coincide everywhere.
\smallskip

5. If $\ulim G_n^{\RLs}$ is a topological group, then the identity map $\glim G_n\to\ulim G_n^{\RLs}$ is continuous because its restrictions to the groups $G_n$ are continuous. The inverse identity  map $\ulim G_n^{\RLs}\to\glim G_n$ is continuous because it is uniformly continuous as the identity map into the topological group $(\glim G_n)^{\RLs}$ endowed with the Roelcke uniformity.

If the identity map $\ulim G_n^{\RLs}\to\glim G_n$ is a homeomorphism, then $\ulim G_n^{\RLs}$ is a topological group because $\glim G_n$ is a topological group.
\end{proof}

\section{Open Problems}\label{s:OP}

Summing up, we conclude that for any tower of topological groups $(G_n)_{n\in\w}$
\begin{itemize}
\item the direct limit $\glim G_n$ is a topological group,
\item $\tlim G_n$ and $\RL[G]$ are quasitopological groups,
\item $\LA[G]=\ulim G_n^{\Ls}$ and $\RA[G]=\ulim G_n^{\Rs}$ are semiuniform groups, and
\item $\LR[G]$, $\ulim G_n^{\LRs}$ and $\ulim G_n^{\RLs}$ are quasiuniform groups, 
\end{itemize}
having the union $G=\bigcup_{n\in\w}G_n$ as their underlying group.

The interplay between these semitopological and semiuniform groups are described in the following diagram.  A simple (resp. double) arrow indicates that the corresponding identity map is continuous (resp. uniformly continuous).

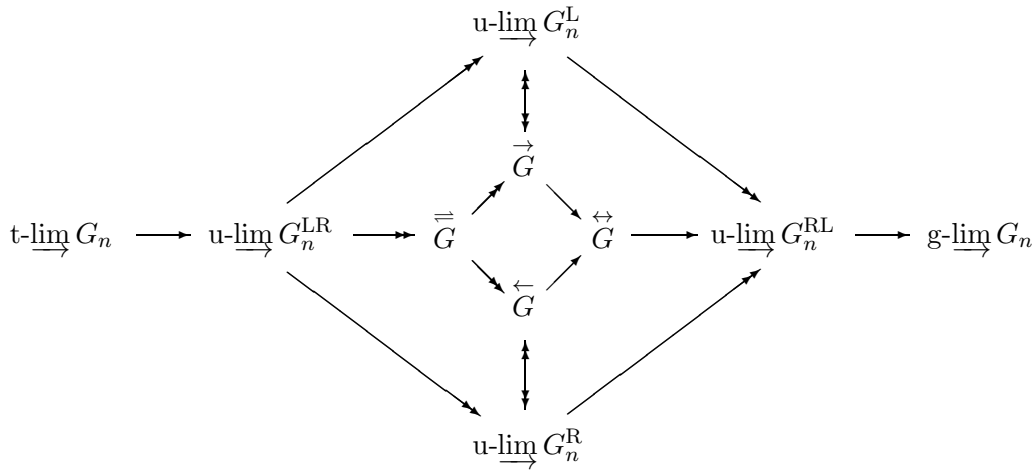
\begin{picture}(300,200)(-220,-100)
\put(-35,-2){$\LR[G]$}
\put(-65,3){\vector(1,0){23}}
\put(-65,3){\vector(1,0){20}}
\put(-20,11){\vector(1,1){12}}
\put(-20,11){\vector(1,1){10}}
\put(-20,-7){\vector(1,-1){12}}
\put(-20,-7){\vector(1,-1){10}}
\put(8,-18){\vector(1,1){13}}
\put(8,22){\vector(1,-1){13}}

\put(25,-2){$\RL[G]$}
\put(40,3){\vector(1,0){25}}
\put(-5,-28){$\RA[G]$}
\put(-5,25){$\LA[G]$}
\put(0,45){\vector(0,1){20}}
\put(0,45){\vector(0,1){17}}
\put(0,65){\vector(0,-1){20}}
\put(0,65){\vector(0,-1){23}}

\put(-120,0){$\ulim G_n^{\LRs}$}
\put(-147,3){\vector(1,0){20}}

\put(-90,15){\vector(4,3){74}}
\put(-90,15){\vector(4,3){71}}
\put(-90,-11){\vector(4,-3){73}}
\put(-90,-11){\vector(4,-3){70}}
\put(16,70){\vector(4,-3){73}}
\put(16,70){\vector(4,-3){70}}
\put(16,-65){\vector(4,3){73}}
\put(16,-65){\vector(4,3){70}}

\put(70,0){$\ulim G_n^{\RLs}$}
\put(125,3){\vector(1,0){20}}

\put(-195,0){$\tlim G_n$}
\put(152,0){$\glim G_n$}
\put(-20,80){$\ulim G_n^{\Ls}$}
\put(-20,-80){$\ulim G_n^{\Rs}$}

\put(0,-62){\vector(0,1){22}}
\put(0,-62){\vector(0,1){25}}
\put(0,-40){\vector(0,-1){20}}
\put(0,-40){\vector(0,-1){23}}
\end{picture}

Under certain conditions on the tower $(G_n)_{n\in\w}$ some of the identity maps in this diagram are homeomorphisms. In particular,
\begin{itemize}
\item $\tlim G_n\to\glim G_n$ is a homeomorphism if all topological groups $G_n$, $n\in\w$, are locally compact  \cite{TSH};
\item $\ulim G_n^{\LRs}\to\glim G_n$ is a homeomorphism if all topological groups $G_n$, $n\in\w$ are balanced \cite{BaR};
\item $\LR[G]\to\glim G_n$ is a homeomorphism if the tower $(G_n)_{n\in\w}$ is balanced or satisfies $\PTA$ (Theorems~\ref{t4.2} and \ref{t3.2});
\item $\RL[G]\to\glim G_n$ is a homeomorphism if the tower $(G_n)_{n\in\w}$ is bi-balanced (Theorem \ref{t5.2}).
\end{itemize}

Nonetheless many open questions related to this diagram remain unsolved. 

\begin{problem} Is the identity map $\LR[G]\to \ulim G_n^{\LRs}$ (uniformly) continuous?
\end{problem} 

\begin{problem} What can be said about separation properties of the quasitopological group $\RL[G]$? Is it always Tychonoff? Is the identity map $\RL[G]\to\ulim G_n^{\RLs}$ a homeomorphism?
\end{problem}

We define a topological space $X$ to be {\em Tychonoff} if for each closed subset $F\subset X$ and each point $x\in X\setminus F$ there is a continuous function $f:X\to\IR$ with $f(x)=1$ and $f(F)\subset\{0\}$. It is known that each uniform (not necessarily separated) space is Tychonoff. In particular, each semiuniform group is Tychonoff.

Surprisingly, but we know no (natural) example of a tower of topological groups $(G_n)_{n\in\w}$ for which the topology of $\glim G_n$ would be different from $\RL[\tau]$ or even $\LR[\tau]$. However, we expect counterexamples to the following problem.

\begin{problem} Is the identity map $\ulim G_n^{\RLs}\to\glim G_n$ a homeomorphism?
What about the identity map  $\ulim G_n^{\LRs}\to\glim G_n$?
\end{problem}

\section{Acknowledgment}

The authors express their thanks to the referee for many valuable remarks and suggestions improving the presentation.

\end{document}